\documentclass{article}
\usepackage{bbm}
\usepackage[T2A]{fontenc}
\usepackage[english]{babel}
\usepackage{amsthm,amsfonts,amsbsy, amssymb,amsmath}
\usepackage{graphicx,graphics}
\usepackage{mathrsfs}
\usepackage[all]{xy}
\usepackage{xcolor}

\newtheorem{theorem}{Theorem}
\newtheorem{lemma}{Lemma}
\newtheorem{corollary}{Corollary}
\newtheorem{proposition}{Proposition}

\theoremstyle{definition}
\newtheorem{definition}{Definition}
\newtheorem{remark}{Remark}
\newtheorem{example}{Example}

\newcommand{\Z}{\mathbb Z}
\newcommand{\N}{\mathbb N}

\newcommand{\bega}{\left(\begin{array}}
\newcommand{\ena}{\end{array}\right)}

\newcommand{\skcrv}{\raisebox{-0.25\height}{\includegraphics[width=0.5cm]{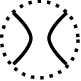}}}
\newcommand{\skcrh}{\raisebox{-0.25\height}{\includegraphics[width=0.5cm]{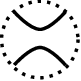}}}
\newcommand{\skcrr}{\raisebox{-0.25\height}{\includegraphics[width=0.5cm]{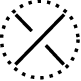}}}

\newcommand{\skcrro}{\raisebox{-0.25\height}{\includegraphics[width=0.5cm]{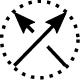}}}
\newcommand{\skcrlo}{\raisebox{-0.25\height}{\includegraphics[width=0.5cm]{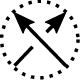}}}
\newcommand{\skcrvo}{\raisebox{-0.25\height}{\includegraphics[width=0.5cm]{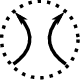}}}
\newcommand{\skcircle}{\raisebox{-0.25\height}{\includegraphics[width=0.5cm]{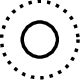}}}
\newcommand{\skempty}{\raisebox{-0.25\height}{\includegraphics[width=0.5cm]{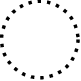}}}
\newcommand{\skcircarc}{\raisebox{-0.25\height}{\includegraphics[width=0.5cm]{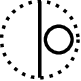}}}
\newcommand{\skarc}{\raisebox{-0.25\height}{\includegraphics[width=0.5cm]{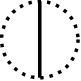}}}

\title{On skein invariants}

\author{Igor Nikonov\footnote{Lomonosov Moscow State University, Moscow, 119991 Russia\\ Email: nikonov at mech.math.msu.su }}

\date{}

\begin{document}

\maketitle

\begin{abstract}
    A knot invariant is called skein if it is determined by a finite number of skein relations. In the paper we discuss some basic properties of skein invariants and mention some known examples of skein invariants.
\end{abstract}



\section{Introduction}

In the late 1960s John Conway~\cite{Conway} noted that the renown Alexander polynomial $\nabla$ can be defined using a simple relation on values of the polynomial on diagrams which differ locally from each other:
\begin{gather*}
\nabla(\skcrro)-\nabla(\skcrlo)=z\nabla(\skcrvo).
\end{gather*}
Conway used the term \emph{skein} for this type of relations.

In fifteen years, other polynomial invariants (like Jones~\cite{Jones}, HOMFLY-PT~\cite{HOMFLY,PT}) appeared which can also be defined using skein relations. At about the same time, L.H. Kauffman~\cite{KauffmanBook} gave a description of Arf invariant using pass move. It became clear that a new type of invariants had appeared in knot theory. The value of such an invariant of a knot is the class of the knot in some quotient of the set of knots generated by skein relations. The next step was made by J.H. Przytycki~\cite{Prz} and V.G. Turaev~\cite{Tur} who defined the notion of skein module. A review of known results on skein modules can be found in~\cite{BPW}.

The aim of the paper is to recall some known examples of skein invariants and discuss their basic properties. 

The paper is organized as follows. We start with definitions of knot diagrams and moves and skein invariants. In Section~\ref{sect:skein_invariant} we give the formal definition of skein invariants and consider some structures on them. Section~\ref{sect:skein_invariant} contains some examples of skein invariants. We conclude the paper with a list of open problems concerning skein invariants.


\section{Knot diagrams}\label{sect:knot_diagrams}

\begin{definition}\label{def:tangle}
Let $F$ be an oriented compact connected surface. A \emph{tangle diagram} $D$ is an embedded finite graph with vertices of valences $1$ and $4$ such that the set $\partial D$ of vertices of valence $1$ coincides with $D\cap\partial F$ and the vertices of valence $4$ (called \emph{crossings}) carry additional structure: structure of a classical or a virtual crossing (Fig.~\ref{fig:crossing_types}). A diagram without virtual crossings is called \emph{classical}. 

Diagrams are considered up to isotopy fixed on the boundary.

\begin{figure}[h]
\centering\includegraphics[width=0.2\textwidth]{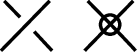}
\caption{A classical and a virtual crossings}\label{fig:crossing_types}
\end{figure}

Edges incident to a $4$-valent vertex of a tangle diagram split naturally in two pairs of \emph{opposite} edges. Correspondence between opposite edges induces an equivalence relation on the set of edges of the diagram. Equivalence classes of this relation are called \emph{(unicursal) components} of the diagram. A component is \emph{long} if it contains vertices of valency $1$ (in this case the edges of the component form a path in the diagram), otherwise the component is called \emph{closed} (in this case the edges of the component form a cycle), see Fig.~\ref{fig:tangle}.

We say that a diagram is \emph{oriented} if all its components are oriented.

\begin{figure}[h]
\centering\includegraphics[width=0.3\textwidth]{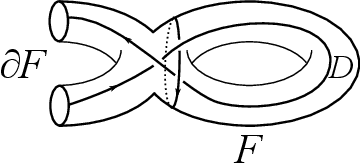}
\caption{A diagram with one closed and one long component}\label{fig:tangle}
\end{figure}

A diagram without long components is a \emph{link} diagram, a link diagram with one component is a \emph{knot} diagram. A diagram with one component which is long, is called a \emph{long knot} diagram.
\end{definition}

\begin{definition}\label{def:n-tangle}
An \emph{$n$-tangle} is a diagram $D$ in the standard disk $\mathbb{D}^2$ such that $\partial D=X$ where $X\subset \mathbb{D}^2$ is a fixed counterclockwise enumerated set with $2n$ elements. The set of $n$-tangles is denoted $\mathcal T_n$, and $\mathcal T_n^+$ is the set of oriented $n$-tangles. An $n$-tangle may have crossings of any type (classical or virtual, Fig.~\ref{fig:crossing_types}).

\begin{figure}[h]
\centering\includegraphics[width=0.12\textwidth]{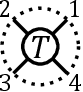}
\caption{A 2-tangle}\label{fig:2tangle}
\end{figure}
\end{definition}

\begin{definition}\label{def:local_move}
A \emph{local move} is a pair $M=(T_1,T_2)$ of $n$-tangles such that $\partial T_1=\partial T_2$.
\end{definition}

Given a move $M=(T_1,T_2)$, a diagram $D$ and a disk $B\subset F$ such that $T=D\cap B$ is homeomorphic to $T_1$, one gets a new diagram $R_M(D,T)$ by replacing the subtangle $T$ with the subtangle homeomorphic to $T_2$.

\begin{figure}[h]
\centering\includegraphics[width=0.35\textwidth]{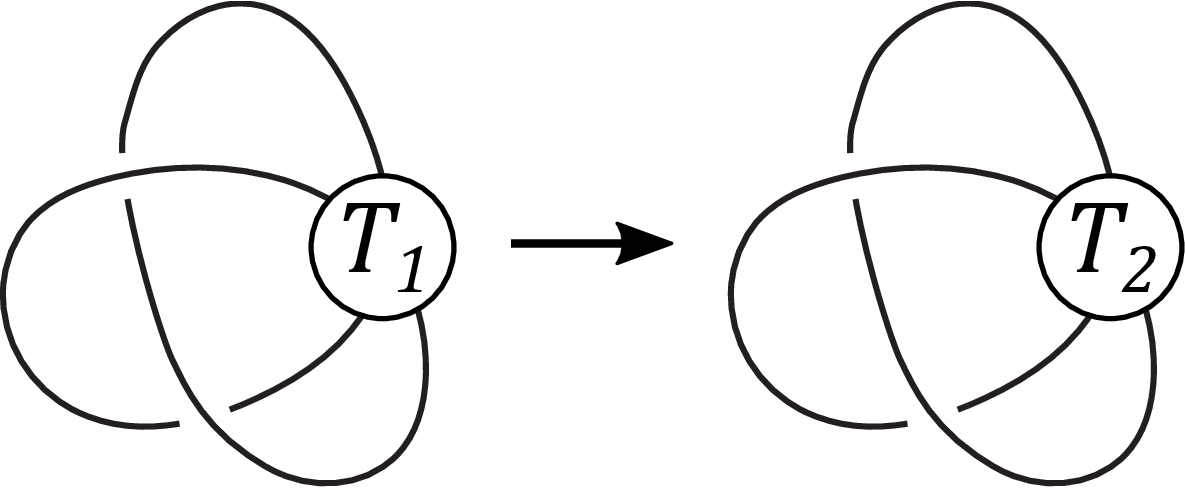}
\caption{Application of a local move $(T_1,T_2)$ to a diagram}\label{fig:local_transformation}
\end{figure}

\begin{remark}
We will also consider local moves $(T_1,T_2)$ with $T_1,T_2\in A[\mathcal T_n]$ where $A$ is a coefficient ring. In this case the result of the local move applied to a diagram is a linear combination of diagrams with coefficients in $A$. The moves $(T_1,T_2)$ with $T_1,T_2\in \mathcal T_n$ are called \emph{monomial}.
\end{remark}

We can use the notion of moves to give a combinatorial definition of knots.

\begin{definition}\label{def:knot_theory}
A {\em (classical) knot} (link, tangle) in a surface $F$ is an equivalence class of classical knot (link, tangle) diagrams in $F$ modulo the classical Reidemeister moves $\Omega_1, \Omega_2, \Omega_3$, see Fig.~\ref{fig:reidemeister_moves}.

A {\em virtual knot} (link, tangle) is an equivalence class of virtual diagrams in $\mathbb R^2$ modulo the classical and virtual Reidemeister moves $\Omega_1$, $\Omega_2$, $\Omega_3$, $V\Omega_1$, $V\Omega_2$, $V\Omega_3$, $SV\Omega_3$.

A {\em welded knot} (link, tangle) is an equivalence class of virtual diagrams in $\mathbb R^2$ modulo the classical and virtual Reidemeister moves and the forbidden move $F^o$.

\begin{figure}[h]
\centering\includegraphics[width=0.3\textwidth]{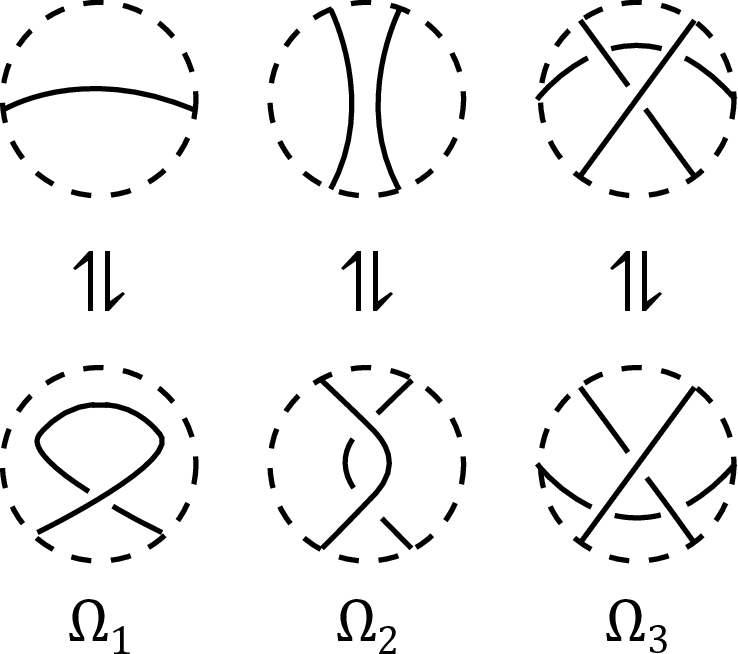}\qquad \includegraphics[width=0.415\textwidth]{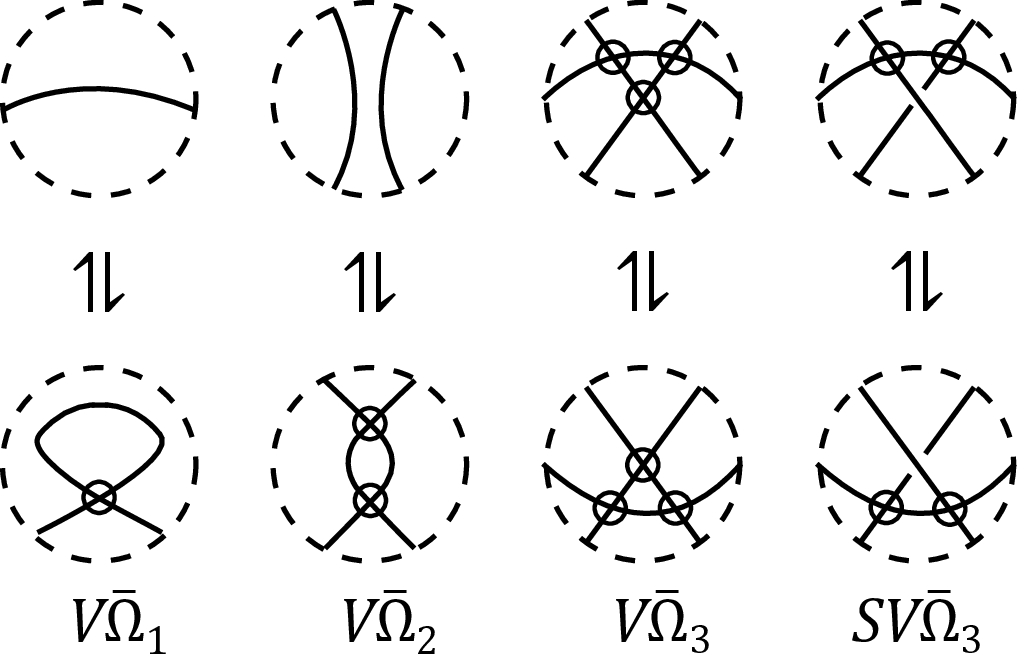}
\qquad \includegraphics[width=0.085\textwidth]{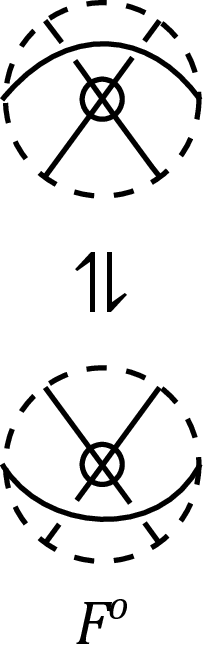}
\caption{Classical and virtual Reidemeister moves}\label{fig:reidemeister_moves}
\end{figure}
\end{definition}

\section{Skein invariants}\label{sect:skein_invariant}

\begin{definition}
Let $\mathscr K$ be the set of classical (virtual, welded) knots (links, tangles), $A$ a ring and $\mathcal M=\{(T_i,T'_i)\}_{i\in\mathcal I}$, where $T_i, T'_i\in A[\mathcal T_{n_i}]$, $n_i\in\mathbb N$, $i\in\mathcal I$, a set of moves.  Given two $A$-linear combinations of diagrams $D$ and $D'$, we say that $D$ and $D'$ are \emph{equivalent modulo the set of moves $\mathcal M$} if there exists a sequence of $A$-linear combinations of diagrams $D=D_1, D'_1,\dots, D_l,D'_l=D'$ such that $D_k$ and $D'_k$, $k=1,\dots,l$, differ by Reidemeister moves and diagram isotopies (i.e. define the same $A$-linear combination of knots in $A[\mathscr K]$), and $D_{k+1}$ is the result of application of some local move $M\in\mathcal M$ to $D'_k$, $k=1,\dots, l-1$. For equivalent $D$ and $D'$ we use the notation $D\sim_{\mathcal M} D'$.

The set $A[\mathscr K]/\mathcal M$ of equivalence classes of $A$-linear combinations of diagrams modulo $\mathcal M$ is called the \emph{skein module} induced by the set of moves $\mathcal M$. 
\end{definition}

\begin{definition}\label{def:skein_invariant}
Let $\mathscr K$ be the set of knots (links, tangles) and $A$ a ring. A knot invariant $I\colon \mathscr K\to X$ with values in a set $X$ is called a \emph{skein invariant} if there exists a finite set of moves $\mathcal M$ and an injection $f\colon A[\mathscr K]/\mathcal M\to X$ such that $I=f\circ p$, where $p\colon\mathscr K\to  A[\mathscr K]/\mathcal M$ is the natural projection.
\end{definition}

\begin{remark}
If the moves in $\mathcal M$ are monomial then in Definition~\ref{def:skein_invariant} we can consider the map $f\colon \mathscr K/\mathcal M\to X$.
\end{remark}

The first examples of skein invariants were polynomial invariants of oriented classical links in $\mathbb R^2$.

\begin{example}[Polynomial skein invariants]\label{ex:polynomial_invariants}

\begin{enumerate}
\item For an oriented link $L$ its Conway polynomial $\nabla(L)\in\mathbb Z[z]$ is determined by the skein relations~\cite{Conway}
\[
\skcrro-\skcrlo \rightleftharpoons z\cdot\skcrvo,\quad \skcircarc \rightleftharpoons 0\cdot\skarc.
\]
\item The Jones polynomial of an oriented link $L$ is $X(L)=(-a)^{-3w(L)}\langle L\rangle\in\mathbb Z[a,a^{-1}]$, where $w(L)$ is the writhe of the link and  $\langle L\rangle$ is the Kauffman bracket~\cite{KauffmanBook}. The Kauffman bracket is determined by the skein relations
\[
\skcrr \rightleftharpoons a\cdot\skcrv+a^{-1} \skcrh,\quad \skcircle \rightleftharpoons (-a^2-a^{-2})\cdot\skempty.
\]
\item The HOMFLY--PT polynomial $PT(L)\in\mathbb Z[l,l^{-1},m,m^{-1}]$ of an oriented link $L$ is determined by the skein relations~\cite{HOMFLY,PT}.
\[
l\skcrro+l^{-1}\skcrlo \rightleftharpoons m\cdot\skcrvo,\quad \skcircarc \rightleftharpoons \frac{l+l^{-1}}m\cdot\skarc.
\]
\end{enumerate}

Below we will focus on monomial moves on knots.
\end{example}

\subsection{Existence of non-skein knot invariants}

First of all, note that any invariant is ``infinitely skein''.

\begin{proposition}\label{prop:infinite_skein}
For any knot invariant $I\colon\mathscr K\to X$ there exists a set of moves $\mathcal M$ and an injection $f\colon \mathscr K/\mathcal M\to X$ such that $I=f\circ p$, where $p\colon \mathscr K\to\mathscr K/\mathcal M$ is the natural projection.
\end{proposition}

\begin{proof}
Indeed, we can take $\mathcal M=\{(K_1,K_2)\in\mathscr K\times\mathscr K \mid I(K_1)=I(K_2)\}$.
\end{proof}

Do there exist non-skein invariants? Yes. The argument is as follows. The set of moves as well as the set of skein invariants is countable. But the set of knot invariants has the cardinality of the continuum. 

Let us give a more concrete construction.

\begin{example}[Non-skein invariant]\label{ex:nonskein_invariant}
Let $\{(T_i,T'_i)\}_{i\in\mathbb N}$ be the set of all local moves such that $\partial T_i=\partial T'_i\ne\emptyset$. Let us construct two sequences of knots $\{\tilde K_i\}$ and $\{\tilde K'_i\}$. 

Consider a move $(T_i,T'_i)$. If all the closures $T_i\# S$ and $T'_i\# S$ of the tangles are equal in $\mathscr K$ then skip the move. Otherwise, take nonequivalent closures $K_i=T_i\# S$ and $K'_i=T'_i\# S$. There exists a prime knot $K''_i$ such that the knots $\tilde K_i=K_i\#K_i''$ and $\tilde K'_i=K'_i\#K_i''$ differ from the knots $\tilde K_j$, $\tilde K'_j$, $j<i$.

Then any invariant $I\colon \mathscr K\to \Z_2$ such that $f(\tilde K_i)=0$ and $f(\tilde K'_i)=1$, is not skein.
\end{example}

\subsection{Degree of moves}\label{sect:moves_order_filtration}

\begin{definition}\label{def:move_order}
A set of moves $\mathcal M$ has \emph{degree $n$} if for any move $M=(T,T')\in\mathcal M$ the tangles $T$ and $T'$ are $n$-tangles.

A skein invariant $I$ has \emph{degree $n$} if it can be determined by a set of moves of degree $n$. Denote the set of skein invariant $I$ of degree $n$ by $\mathscr I_n$.
\end{definition}

Let us describe skein invariants of small degrees.

\subsubsection{Skein invariants of degree $0$}

\begin{proposition}\label{prop:order_0}
Let $\mathcal M$ be a finite set of moves and $\bar{\mathcal M}\subset \mathcal M$ the subset of moves of degree $>0$. Let $I_{\mathcal M}\colon \mathscr K\to\mathscr K/\mathcal M=X_{\mathcal M}$ and $I_{\bar{\mathcal M}}\colon \mathscr K\to \mathscr K/\bar{\mathcal M}=X_{\bar{\mathcal M}}$ be the corresponding skein invariants. Then there exist finite distinct subsets $S_1,\dots, S_k\subset X_{\bar{\mathcal M}}$ such that for any knots $K$, $K'$ one has $I_{\mathcal M}(K)=I_{\mathcal M}(K')$ if and only if $I_{\bar{\mathcal M}}(K)=I_{\bar{\mathcal M}}(K')$ or $I_{\bar{\mathcal M}}(K),I_{\bar{\mathcal M}}(K')\in S_i$ for some $1\le i\le k$.
\end{proposition}

\begin{proof}
    Since $\bar{\mathcal M}\subset \mathcal M$, there is a natural surjection $p\colon X_{\bar{\mathcal M}}\to X_{\mathcal M}$. The map $p$ is determined by the set $\mathcal M\setminus\bar{\mathcal M}=\{(K_i,K'_i)\}$ of moves of degree $0$, where $K_i$ and $K'_i$ are some knots. These moves glue the values $I_{\bar{\mathcal M}}(K_i)$ and $I_{\bar{\mathcal M}}(K'_i)\in X_{\bar{\mathcal M}}$. Then we set the subsets $S_i$ to be the preimages $p^{-1}(x)$, $x\in X_{\mathcal M}$, which contain two or more elements.
\end{proof}

\begin{corollary}
Let $I\in\mathscr I_0$. Then there exist finite distinct subsets $S_1,\dots, S_k\subset \mathscr K$ such that for any knots $K$, $K'$ one has $I(K)=I(K')$ if and only if $K=K'$ or $K,K'\in S_i$ for some $1\le i\le k$.
\end{corollary}
\begin{proof}
    Indeed, in this case $\bar{\mathcal M}=\emptyset$ and $X_{\bar{\mathcal M}}=\mathscr K$.
\end{proof}

\begin{definition}\label{def:reduced_move_set}
A set of moves $\mathcal M$ is \emph{reduced} if $\mathcal M=\bar{\mathcal M}$.
\end{definition}

\subsubsection{Additivity of skein invariants}

\begin{definition}\label{def:additive_invariant}
A knot invariant $I\colon \mathscr K\to X$ is called \emph{additive} if $X$ has a commutative monoid structure and $I(K\# K')=I(K)+I(K')$ for any $K,K'\in\mathscr K$.
\end{definition}

\begin{theorem}\label{thm:reduced_additive_invariant}
Let $I\colon \mathscr K\to X$ be a reduced skein knot invariant and $X=I(\mathscr K)$. Then $I$ is additive.
\end{theorem}

\begin{lemma}\label{lem:additive_invariant}
Let $\mathcal M$ be a reduced set of moves. Then $K_1\sim_{\mathcal M} K_2$ implies $K_1\# K'\sim_{\mathcal M} K_2\# K'$ for any $K'\in\mathscr K$.
\end{lemma}
\begin{proof}
    The condition $K_1\sim_{\mathcal M} K_2$ means there is a sequence of diagrams $K_1=\tilde K_1, \tilde K'_1,\dots, \tilde K_l,\tilde K'_l=K_2$ such that the diagrams $\tilde K_i$ and $\tilde K'_i$ are connected by a sequence of Reidmeister moves, and $\tilde K'_i$ and $\tilde K_{i+1}$ are connected by a move $M_i\in\mathcal M$. Since the move $M_i$ has positive degree, there is an arc in $\tilde K'_i$ which lies outside the move. We place the knot $K'$ on this arc and get diagrams $\hat K'_i$ and $\hat K_{i+1}$ connected by the move $M_i$. Since the diagrams $\tilde K_i$ and $\tilde K'_i$ present the same knot, the diagrams $\hat K_i=\tilde K_i\# K'$ and $\hat K'_i=\tilde K_i\# K'$ are connected by a sequence of Reidemester moves. Thus, the sequence $K_1\# K'=\hat K_1, \hat K'_1,\dots \hat K_l, \hat K'_l=K_2\# K'$ establishes the equivalence $K_1\# K'\sim_{\mathcal M} K_2\# K'$.    
\end{proof}

\begin{proof}[Proof of Theorem~\ref{thm:reduced_additive_invariant}]
Let $\mathcal M$ be a reduced set of moves which determines the invariant $M$. Then the coefficient set $X$ is identified with the set $\mathscr K/\mathcal M$ of equivalence classes of knots modulo the moves of $\mathcal M$. Define a binary operation on $X$ by the formula $[K_1]+[K_2]=[K_1\# K_2]$, $K_1,K_2\in\mathscr K$. The operation is well defined, because for any $K_1\sim_{\mathcal M} K_2$ and $K'_1\sim_{\mathcal M} K'_2$ we have $K_1\# K_2\sim_{\mathcal M} K_1\# K'_2\sim_{\mathcal M} K'_1\# K'_2$ by Lemma~\ref{lem:additive_invariant}. The operation defines on $X$ a monoid structure because the operation $\#$ of connected sum does. The equality $I(K\# K')=I(K)+I(K')$ follows from the definition of the binary operation. 
\end{proof}

\begin{remark}
The set of additive knot invariants has the cardinality of the continuum. Hence, there are non-skein additive invariants.
\end{remark}

\subsubsection{Skein invariants of degree $1$}

\begin{example}\label{ex:order_1_invariant}
Let $P$ be a finite subset of the set $\mathscr P$ of prime knots and $\phi\colon P\to M$ be a map to a commutative monoid $M$. Let us define a knot invariant $I_{P,\phi}\colon \mathscr K\to M\times \Z[\N]$ as follows. Enumerate the prime knots $\mathscr P=\{ K_1, K_2,\dots\}$ so that $P=\{K_1,\dots, K_n\}$. For a knot $K=\sum_i a_i K_i\in\mathscr K$ set
\[
I_{P,\phi}=\left(\sum_{i=1} a_i\phi(K_i),\sum_{i=n+1}^\infty a_i\cdot[i-n]\right).
\]
Then $I_{P,\phi}\in\mathscr I_1$.
Indeed, consider the knots in $P$ as long knots. The map $\phi$ extends to a monoid homomorphism $\phi\colon\mathbb Z_{\ge 0}[P]\to M$. The image $\mathrm{im}\phi$ is finitely generated. By Redey's theorem~\cite{RG} $\mathrm{im}\phi$ is finitely presented, hence, $\mathrm{im}\phi\simeq \mathbb Z_{\ge 0}[P]/\sim$ where the equivalence $\sim$ is generated by a finite set of relations $T_i\sim T'_i$. The linear combinations $T_i,T'_i\in \mathbb Z_{\ge 0}[P]$ can be identified with connected sums of prime knots from $P$.  Then $I_{P,\phi}$ is generated by the set of moves 
$\mathcal M=\{(T_i,T'_i)\}\subset \mathscr I_1.$
\end{example}

\begin{proposition}\label{prop:order_1_invariant}
Let $I\in\mathscr I_1$. Then $I$ is equivalent to the invariant $I_{P,\phi}$ for some finite $P\subset\mathscr P$ and $\phi\colon P\to M$  in the following sense: there is a bijection $f\colon \mathrm{im} I_{P,\phi}\to \mathrm{im} I$ such that $I=f\circ I_{P,\phi}$.
\end{proposition}

\begin{proof}
   The invariant $I$ is determined by a set $\mathcal M=\{(T_i,T'_i)\}_{i=1}^l$ of moves of degree $1$. Then the tangles $T_i,T'_i$ are long knots. Let $P=\{K_1,\dots,K_n\}$ be the set of prime knots which divide one of these tangles and $A=\mathbb Z_{\ge 0}[P]$. Let $\sim$ be the congruence in $A$ generated by the relations $T_i\sim T'_i$ and  $M=A/\sim$ the quotient monoid. Denote the natural projection $P\to M$ by $\phi$. Then $I$ is equivalent to the skein invariant $I_{P,\phi}$. 
\end{proof}


\subsubsection{Stabilization of degree}

\begin{proposition}\label{prop:order_stabilization}
For any $n\ge 1$ $\mathscr I_n\subset \mathscr I_{n+1}$.
\end{proposition}

\begin{proof}
Given a move $M=(T,T')$ of degree $n>0$, we can assign to it a move $M^+$ of degree $n+1$ called the \emph{stabilization of the move $M$} (Fig.~\ref{fig:move_stabilization}). It is easy to see that the move $M$ implies the move $M^+$.

\begin{figure}[h]
\centering\includegraphics[width=0.25\textwidth]{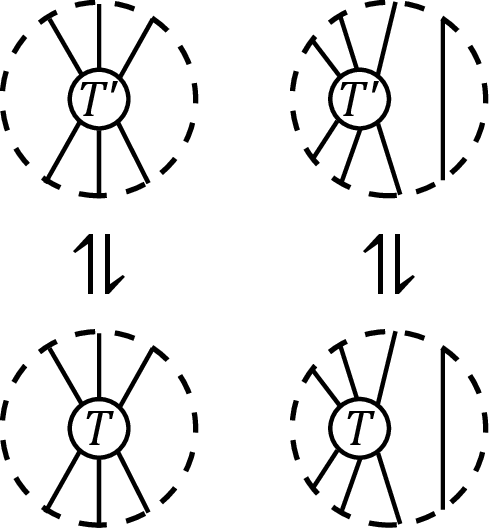}
\caption{Stabilization of a move}\label{fig:move_stabilization}
\end{figure}

On the other hand, the move $M$ can be expressed using the move $M^+$ (Fig.~\ref{fig:move_destabilization}).

\begin{figure}[h]
\centering\includegraphics[width=0.4\textwidth]{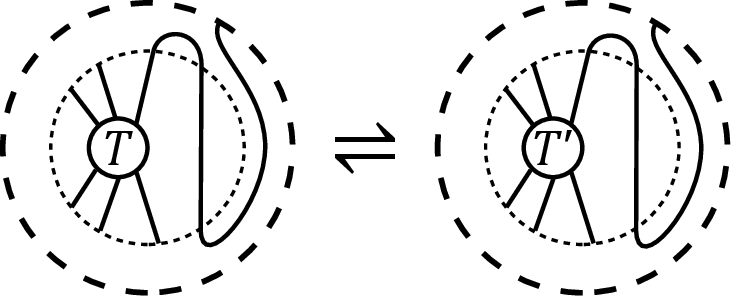}
\caption{Destabilization of a move}\label{fig:move_destabilization}
\end{figure}

Let $\mathcal M=\{M_i\}\subset \mathscr I_n$. Consider the set of moves $\mathcal M^+=\{M^+_i\}\subset\mathscr I_{n+1}$. Then the moves of $\mathcal M$ can be expressed by moves of $\mathcal M^+$ and vice versa. Hence, the skein invariants induced by the sets $\mathcal M$ and $\mathcal M^+$ are equivalent. Thus,  $\mathscr I_n\subset \mathscr I_{n+1}$.
\end{proof}

Thus, we have an increasing filtration of skein invariants: 
\[
\mathscr I_1\subset \mathscr I_{2}\subset \mathscr I_{3}\subset\cdots.
\]
A natural question is whether this filtration stabilizes.

\subsection{Partial order on skein invariants}\label{sect:partial_order}

There is a natural partial order on the set of skein invariants. 

\begin{definition}\label{def:partial_order}
A knot invariant $I\colon\mathscr K\to X$ is \emph{weaker} than a knot invariant $I'\colon\mathscr K\to X'$ if there is a map $f\colon X'\to X$ such that $I=f\circ I'$. In this case we write $I\preceq I'$.

A set of moves $\mathcal M'$ is \emph{finer} than a set of moves $\mathcal M$ if any move $M'\in\mathcal M'$ can be expressed as a sequence of moves from $\mathcal M$ and Reidemeister moves. In this case we write $\mathcal M\preceq\mathcal M'$.
\end{definition}

\begin{proposition}\label{prop:partial_order}
1. If $\mathcal M\preceq\mathcal M'$ then $I_{\mathcal M}\preceq I_{\mathcal M'}$.\\
2. If $I\preceq I'$ are skein invariants then there exist sets of moves $\mathcal M\preceq\mathcal M'$ such that $I= I_{\mathcal M}$ and  $I'= I_{\mathcal M'}$.
\end{proposition}
\begin{proof}
    The first statement follows from definitions.

    Let $I\preceq I'$ and $I=I_{\mathcal M_0}$ and $I'=I_{\mathcal M'}$. Denote $\mathcal M=\mathcal M_0\cup\mathcal M'$. Then $I_{\mathcal M_0}$ is equivalent to $I_{\mathcal M}$, and $\mathcal M\preceq \mathcal M'$.
\end{proof}

\begin{remark}
1. $I_{\mathcal M}\preceq I_{\mathcal M'}$ does not imply $\mathcal M\preceq\mathcal M'$. For example, consider the unknotting moves $\mathcal M=\{\Delta\}$, $\mathcal M'=\{CC\}$ (Fig.~\ref{fig:unknotting_moves}).

2. $\mathscr I$ is a lower semilattice (with the meet operation given by the union of sets of moves).
\end{remark}

\subsubsection{Unknotting moves}

The weakest skein invariant is the trivial invariant. It can be induced by different sets of moves. 

\begin{definition}\label{def:unknotting_move}
A set of moves $\mathcal M$ is \emph{unknotting} if $I_{\mathcal M}$ is trivial.
\end{definition}

\begin{example}[Unknotting moves for classical knots]
Each of the moves in Fig.~\ref{fig:unknotting_moves} is unknotting for classical knots in $\mathbb R^2$~\cite{Aida,AYS,HNT,Matveev,Murakami,MN,Shibuya}.

\begin{figure}[h]
\centering\includegraphics[width=0.9\textwidth]{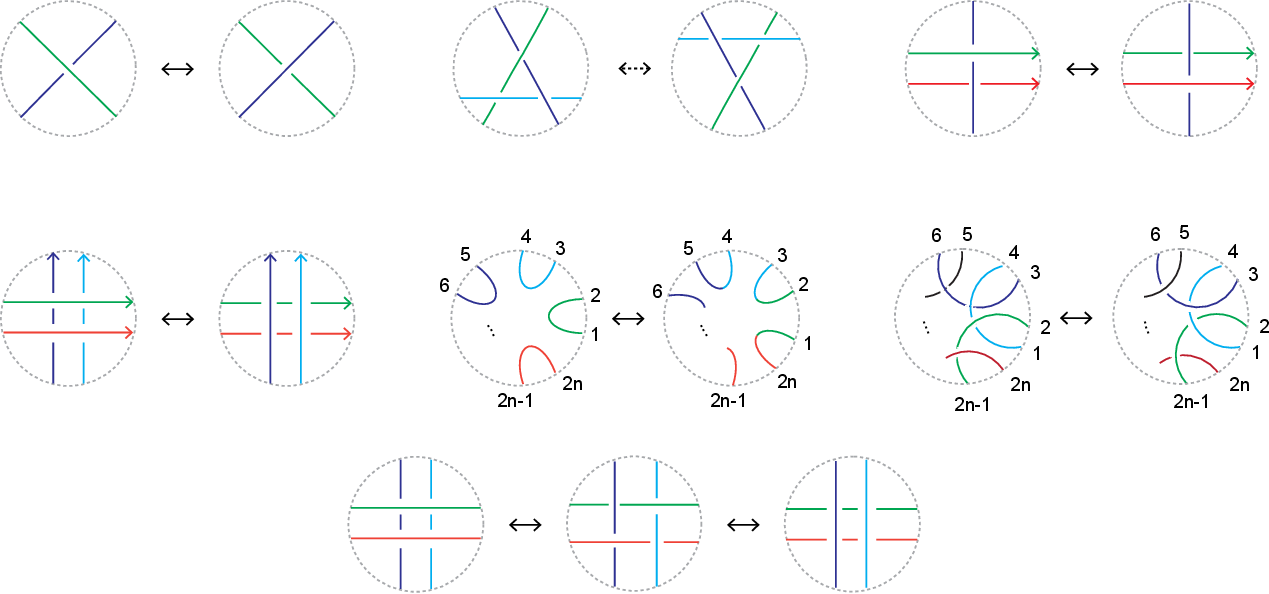}
\caption{Unknotting moves: $CC$-move (crossing change), $\Delta$-move~\cite{Matveev,MN}, $\Gamma$-move~\cite{Shibuya}; $\sharp$-move~\cite{Murakami}, $H(n)$-move~\cite{HNT}, $n$-gon move~\cite{Aida}; and the diagonal move~\cite{AYS}}\label{fig:unknotting_moves}
\end{figure}    
\end{example}

\begin{example}[Unknotting moves for virtual knots]
The following sets of moves are unknotting for virtual knots:
\begin{itemize}
    \item the forbidden move $F^m$~\cite{YI} (Fig.~\ref{fig:forbidden_moves1});
    \item the CF-move~\cite{Oikawa};
    \item the pair of forbidden moves $\mathcal M=\{F^o, F^u\}$~\cite{Kanenobu,Nelson};
    \item any of the virtualizing $\Delta$-move, virtualized $\sharp$-move, virtualized pass move~\cite{NNSW} (Fig.~\ref{fig:virtual_unknotting_moves}).
\end{itemize}

\begin{figure}[h]
\centering\includegraphics[width=0.3\textwidth]{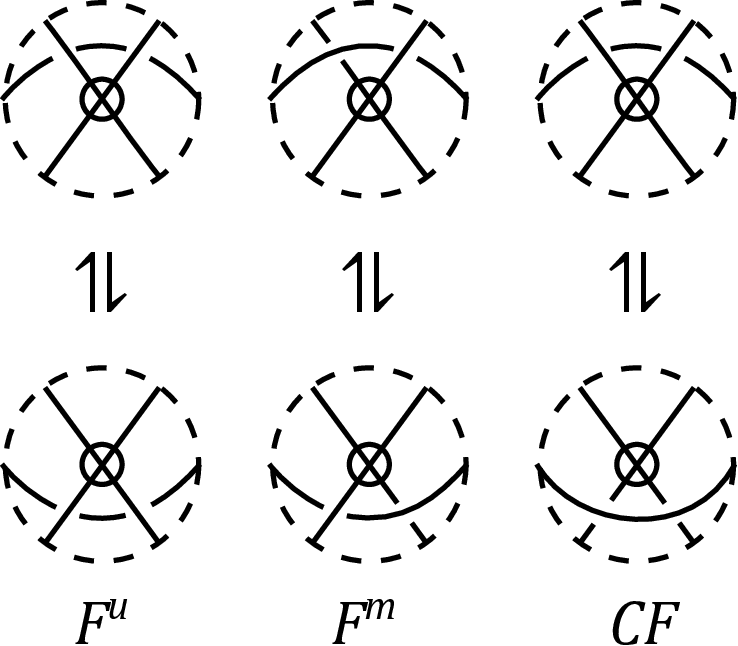}
\caption{Forbidden virtual moves}\label{fig:forbidden_moves1}
\end{figure}

\begin{figure}[h]
\centering\includegraphics[width=0.7\textwidth]{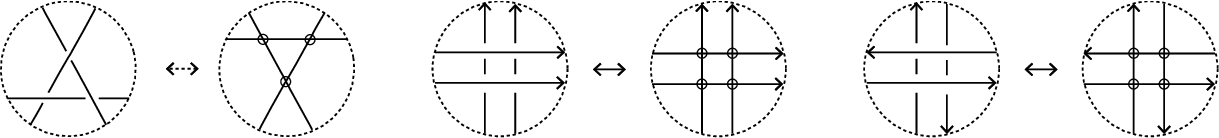}
\caption{Virtualized $\Delta$-move, virtualized $\sharp$-move, virtualized pass move~\cite{NNSW}}\label{fig:virtual_unknotting_moves}
\end{figure}
\end{example}

\begin{example}[Unknotting moves for welded knots]
 Each of $CC$-move, $\Delta$-move, $\sharp$-move and pass move (Fig.~\ref{fig:pass_move}) is unknotting for the welded knots~\cite{Satoh,NNSY}.  
\end{example}

\subsubsection{Binary skein invariants}

\begin{definition}\label{def:binary_invariant}
A nontrivial skein invariant $I$ is a \emph{weakest nontrivial invariant} if any skein invariant $I'\prec I$ is trivial (i.e. a constant map).

A skein invariant $I$ is \emph{binary} if $|\mathrm{im} I|=2$.
\end{definition}

\begin{proposition}\label{prop:binary_invariant}
A skein invariant $I$ is weakest nontrivial if and only if it is binary.
\end{proposition}

\begin{proof}
Let $I$ be binary. If $I'\prec I$ then $|\mathrm{im} I'|<|\mathrm{im} I|=2$. Hence, $|\mathrm{im} I'|=1$, and the invariant $I'$ is trivial.

Let a skein invariant $I=I_{\mathcal M}$ be not binary. Then $I$ is not trivial and there exist knots $K$,$K'$ such that $I(K)\ne I(K')$. Denote $\mathcal M'=\mathcal M\cup \{(K,K')\}$. Then $I_{\mathcal M'}\prec I$ and $|\mathrm{im} I_{\mathcal M'}|=|\mathrm{im} I|-1>1$. Thus, $I_{\mathcal M'}$ is not trivial, and $I$ is not a weakest nontrivial invariant.
\end{proof}

\begin{example}[Arf invariant]\label{ex:arf_invariant}
L.H. Kauffman~\cite{KauffmanBook} showed that the pass move (Fig.~\ref{fig:pass_move}) generates the Arf invariant which takes values in $\mathbb Z_2$.
\begin{figure}[h]
\centering\includegraphics[width=0.2\textwidth]{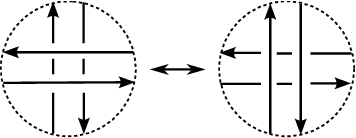}
\caption{Pass move}\label{fig:pass_move}
\end{figure}
\end{example}

\begin{remark}
    Any skein invariant $I\colon\mathscr K\to X$ with a finite value set $X$ is equivalent to the family of binary skein invariants $\mathbf I=\{I_x\}_{x\in X}$ where 
\[
I_x(K)=\left\{\begin{array}{cl}
       1,  & I(K)=x; \\
       0,  & I(K)\ne x.
    \end{array}\right.
\]
\end{remark} 

\subsection{Cardinality of skein invariant}

\begin{definition}\label{def:skein_invariant_cardinality}
The \emph{cardinality} $|\mathcal M|$ of a set of moves $\mathcal M$ is the number of moves in $\mathcal M$.

The \emph{cardinality $|I|$ of a skein invariant} $I\colon \mathscr K\to X$ is the minimal number of moves generating $I$:
\[
|I| = \min\{|\mathcal M| \mid I\sim I_{\mathcal M}\}.
\]
\end{definition}

\begin{remark}
\begin{enumerate}
    \item A skein invariant $I$ has cardinality $0$ if and only if it is a full invariant.
    \item Skein invariants of cardinality $1$ can be various:
    \begin{itemize}
        \item The invariant generated by a $0$-degree move $(K_1,K_2)$, where $K_1, K_2$ are some knots, is almost full: it distinguishes all knots except the pair $K_1, K_2$.
        \item The classical knot invariant generated by the crossing change move $CC$ is trivial.
        \item The invariant generated by the Conway relation 
\[
\skcrro-\skcrlo \rightleftharpoons z\cdot\skcrvo
\]
is equivalent to the Alexander--Conway polynomial.
    \end{itemize}    
\item The cardinality induces an increasing filtration of skein invariant 
\[
\mathscr I'_0\subset\mathscr I'_1\subset\mathscr I'_2\subset\cdots
\]
where $\mathscr I'_n =\{ I\in\mathscr I \mid |I|\le n\}$. This filtration is infinite. Indeed, the $0$-degree invariant $I_{\mathcal M}$ determined by a set of moves $\mathcal M=\{(K_i,K'_i)\}_{i=1}^n$ such that the knots $K_i$, $K'_i$, $i=1,\dots,n$, are all distinct, has cardinality $n$.
\end{enumerate}    
\end{remark}

\section{Other examples of skein invariants}\label{sect:examples}

Let us list some known examples of skein invariants besides those mentioned above.

\subsection{Finite type invariants}

Finite type invariants were introduced by V. Vasiliev~\cite{Vasiliev}. M. Gussarov~\cite{Goussarov} and K. Habiro~\cite{Habiro} showed that finite type invariants are skein.

\begin{theorem}[M. Gussarov, K. Habiro]\label{thm:finite_type_invariant}
For any knots $K,K'\in\mathscr K$ the following conditions are equivalent:
\begin{enumerate}
\item $K\sim_{C_n} K'$;
\item $v(K)=v(K')$ for any finite invariant of order $\le n-1$
\end{enumerate}
In other words, $I_{C_n}$ is the universal finite type invariant of order $\le n-1$.
\end{theorem}
\begin{figure}[h]
\centering\includegraphics[width=0.8\textwidth]{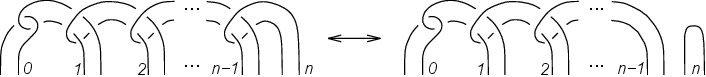}
\caption{$C_n$-move}
\end{figure}
\begin{remark}
The move $C_1$ is equivalent to the crossing change move, the move $C_2$ is equivalent to the $\Delta$-move.
\end{remark}

\subsection{$S$-equivalence}

\begin{definition}\label{def:S-equivalence}
Two square integral matrix are \emph{$S$-equivalent} if they are connected by a sequence of transformations $M\sim PMP^{t}$, $P\in GL(\Z)$, and 
\[
M\sim
\left(\begin{array}{ccc}
M &\mathbf{a} & \mathbf{0}\\
\mathbf{b} & c & 1 \\
\mathbf{0} & 0 & 0
\end{array}\right) \sim
\left(\begin{array}{ccc}
M &\mathbf{a} & \mathbf{0}\\
\mathbf{b} & c & 0 \\
\mathbf{0} & 1 & 0
\end{array}\right).
\]
Classical knots $K$ and $K'$ are \emph{$S$-equivalent} if their Seifert matrices are $S$-equivalent. (For the definition of a Seifert matrix of a knot see~\cite{Rolfsen}.)
\end{definition}

\begin{theorem}[S. Naik, T. Stanford~\cite{NS}]\label{thm:S-equivalence}
Two knots $K$ and $K'$ are $S$-equivalent if and only if they are equivalent by
a sequence of doubled-delta moves.
\end{theorem}
\begin{figure}[h]
\centering\includegraphics[width=0.35\textwidth]{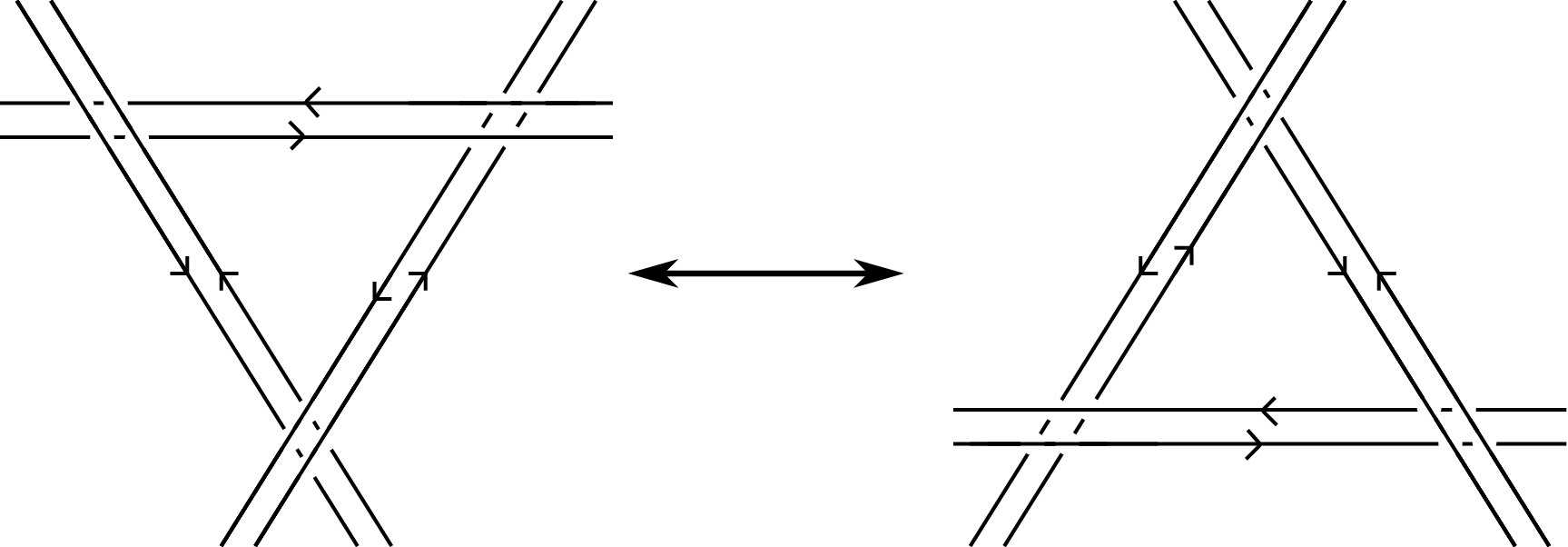}
\end{figure}


\subsection{Linking numbers of welded links}

Audoux et al.~\cite{ABMW} showed that linking numbers of welded links are skein invariants.

\begin{definition}\label{def:linking_number}
Let $L=K_1\cup\cdots\cup K_n$ be a welded link diagram with enumerated components. The \emph{linking number} of the components $K_i$ and $K_j$ is $lk_{ij}(L)=\sum_{c: K_i\mbox{\scriptsize over }K_j}sgn(c)$. The difference $W(K_i,K_j)=l_{ij}-l_{ji}$ is called the \emph{wriggle number} of the pair of components $K_i,K_j$.
\end{definition}

\begin{theorem}[\cite{ABMW}]\label{thm:ABMW}
1. $I_F\simeq (lk_{ij})_{i,j=1}^n$;\\
2. $I_{VC}\simeq (lk_{ij}+lk_{ji})_{i<j}$;\\
3. $I_{CC}\simeq (lk_{ij}-lk_{ji})_{i<j}$;\\
4. $I_{wBP}\simeq (lk^{\Z_2}_{ij}+lk^{\Z_2}_{ji})\oplus(\sum_{j\ne i}lk_{ij}^{\Z_2})_i$.
\end{theorem}

\begin{figure}[h]
\centering\includegraphics[width=0.6\textwidth]{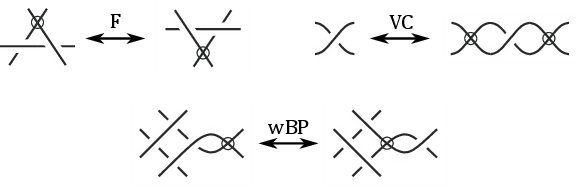}
\caption{Moves on welded knots}
\end{figure}

\subsection{Index polynomial}

Some parity invariants of virtual knots appear to be skein.

\begin{definition}\label{def:gaussian_index}
Let $D$ be a virtual knot diagram and $c$ a crossing of $D$. The oriented smoothing of $D$ at the crossing $c$ splits the diagram into the left and right halves $D_c^l$ and $D_c^r$ (Fig.~\ref{fig:knot_halves}). The \emph{(Gaussian) index} of the crossing $c$ is $ind(c)=sgn(c)W(D_c^r,D_c^l)\in\Z$.

\begin{figure}
\centering\includegraphics[width=0.5\textwidth]{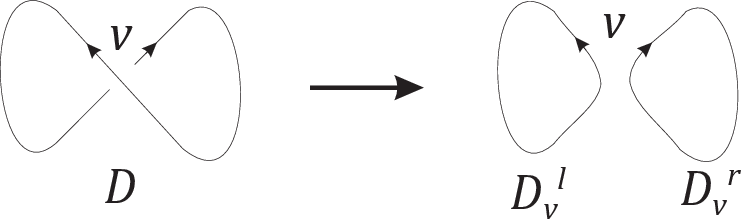}
\caption{Knot halves at a crossing}\label{fig:knot_halves}
\end{figure}

The \emph{odd writhe}~\cite{Kauffman2004} of the virtual knot is the sum $J(D)=\sum_{c: ind(c)\mbox{\scriptsize\ is odd}}sgn(c)$.

The \emph{index polynomial}~\cite{FK} of the virtual knot is 
\[
W_D(t)=\sum_c sgn(c)\cdot(t^{ind(c)}-1).
\]
\end{definition}

Recall that virtual knots can be presented by Gauss diagrams~\cite{GPV}. The Gauss diagram of a virtual knot diagram is an oriented chord diagram whose chords correspond to the crossings of the virtual diagram, and are equipped with a sign (the sign of the crossing) and an orientation (from the overcrossing to the undercrossing), see Fig.~\ref{fig:virtual_gauss_diagram}.

\begin{figure}[h]
\centering\includegraphics[width=0.15\textwidth]{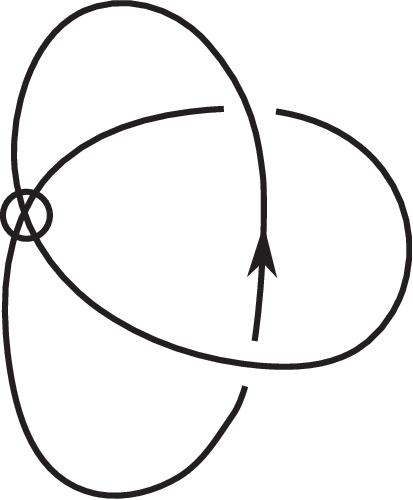}\qquad
\includegraphics[width=0.15\textwidth]{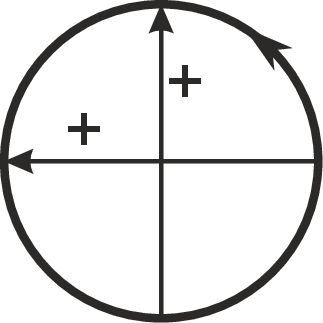}
\caption{A virtual knot and its Gauss diagram}\label{fig:virtual_gauss_diagram}
\end{figure}

Then one can define local moves as local transformations of Gauss diagrams (Fig.~\ref{fig:shell_xi_moves}).

\begin{theorem}[\cite{NNS,ST}]\label{thm:index_polynomial}
Let $K$ and $K'$ be oriented virtual knots. Then
\begin{enumerate}
\item The condition $K\sim_{\Xi} K'$ is equivalent to $J(K)=J(K')$;
\item The condition $K\sim_{\{S_1, S_2\}} K'$ is equivalent to $W_K(t)=W_{K'}(t)$.
\end{enumerate}
\end{theorem}

\begin{figure}
\includegraphics[width=0.6\textwidth]{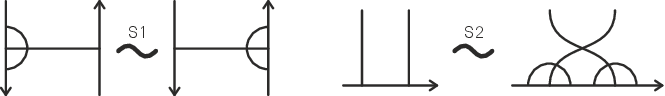}\qquad
\includegraphics[width=0.25\textwidth]{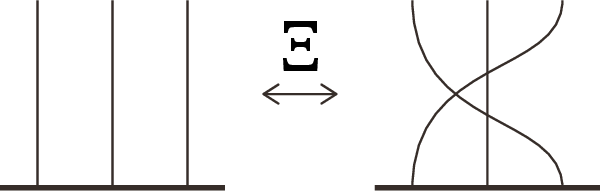}
\caption{Shell moves and $\Xi$-move}\label{fig:shell_xi_moves}
\end{figure}


\section{Open questions}

In the paper we have took first steps towards a theory of skein invariants, and there are a lot of unanswered questions. Let us list some of them.

\begin{itemize}
\item Are known knot invariants such as unknotting number $u(K)$, the crossing number $c(K)$, the genus $g(K)$ skein? Skeinness of the unknotting and the crossing number would imply their additivity.
\item Does the sequence of skein invariants $\mathscr I_1\subset \mathscr I_{2}\subset \mathscr I_{3}\subset\cdots$ stabilize?
\item Find conditions for a set of moves to be unknotting.
\item Describe binary skein invariants.
\item Describe skein invariants of cardinality $1$.
\item Besides order or cardinality, one can filter skein invariants by complexity (number of crossings) of moves. The task is to describe skein invariants of small complexity.
\end{itemize}

\end{document}